\documentclass[twoside,a4paper,11pt]{article}
\usepackage{amsthm,amsmath,amssymb,amscd,graphicx,amsfonts,stmaryrd}
\usepackage{mathrsfs}
\usepackage[all]{xy}
\usepackage[margin=25mm]{geometry}
\usepackage{mathtools}
\usepackage{euscript}
\usepackage{enumitem}
\usepackage{fancyhdr}
\usepackage{latexsym}
\usepackage{amsfonts}
\usepackage{bm}
\usepackage[d]{esvect}
\usepackage{array}
\usepackage{url}
\usepackage{hyperref}
\usepackage{xcolor}
\usepackage[normalem]{ulem}

\allowdisplaybreaks[1]

\thepage

\mathtoolsset{showonlyrefs=true}

\numberwithin{equation}{section}

\theoremstyle{definition}
\newtheorem{exm} {Example}[section]
\newtheorem{dfn}[exm] {Definition}
\newtheorem{rem}[exm] {Remark}

\theoremstyle{plane}
\newtheorem{lem}[exm]{Lemma}
\newtheorem{prop}[exm]{Proposition}
\newtheorem{thm}[exm]{Theorem}

\newtheorem{assum}[exm]{Assumption}

\DeclareMathOperator{\supp}{supp} 
 
\newcommand{\RN}{\mathbb{R}} 

\newcommand{\Borel}{\mathcal{B}}

\newcommand{\domain}{\mathcal{F}}
\newcommand{\form}{\mathcal{E}}
\newcommand{\filt}[0]{\EuScript{F}}
\newcommand{\smooth}[0]{\mathcal{S}}
\newcommand{\resolop}[0]{R}
\newcommand{\resol}[0]{r}
\newcommand{\sigalg}[0]{\EuScript{M}}

\newcommand{\supnorm}[1]{\| #1 \|_{\infty}}
\newcommand{\inftynorm}[2]{\| #1 \|_{L^{\infty}(S, #2)}}

\title{Continuity of the Revuz correspondence under the absolute continuity condition} 
\date{}
\author{Ryoichiro Noda\thanks{Research Institute for Mathematical Sciences, Kyoto University, Kyoto, 606-8502,
JAPAN. E-mail:sgrndr@kurims.kyoto-u.ac.jp}}

\begin{document}

\maketitle
\begin{abstract}
  In this paper, 
  we consider 
  standard processes that admit dual processes and satisfy the absolute continuity condition,
  i.e., processes possess transition densities.
  For such processes, 
  the Revuz correspondence relates positive continuous additive functionals (PCAFs) to so-called smooth measures.
  We show the continuity of this correspondence. 
  Specifically,
  we show that 
  if the $1$-potentials of smooth measures converge (locally) uniformly as functions,
  then the associated PCAFs converge.
  This result is derived by directly estimating the distance between the PCAFs 
  in terms of the distance between the $1$-potentials of the associated smooth measures.
  Furthermore,
  in cases where the transition density is jointly continuous,
  we present sufficient conditions for the convergence of $1$-potentials 
  based on the weak or vague convergence of smooth measures.
  The framework in this paper contains the class of symmetric Hunt processes 
  that are associated with regular Dirichlet forms and satisfy the absolute continuity condition.
\end{abstract}



\section{Introduction} \label{sec: introduction}

Positive continuous additive functionals (PCAFs) of Markov processes, such as local times, 
can be very useful for analyzing these processes (see Definition~\ref{dfn: PCAF} below).
In \cite{Revuz_70_Mesures}, 
Revuz discovered a one-to-one correspondence between PCAFs and measures known as smooth measures (see Definition~\ref{dfn: smooth measure} below), 
which is now called the Revuz correspondence. 
This correspondence forms a core of the theory of PCAFs. 
However, little research has been conducted on the topological properties of the Revuz correspondence.
Recently, 
Nishimori, Tomisaki, Tsuchida and Uemura \cite{Nishimori_Tomisaki_Tsuchida_Uemura_2025} demonstrated a certain compactness property of the Revuz correspondence. 
In particular, they showed that for a symmetric Hunt process associated with a regular Dirichlet form $(\form, \domain)$ on $L^{2}(S, m)$, 
where $S$ is a locally compact, separable, metrizable topological space and $m$ is a Radon measure on $S$ with full support, 
if the $1$-potentials of smooth measures of finite energy integrals converge with respect to $\form_{1}$, 
where $\form_{1}(f, g) \coloneqq \form(f, g) + \int f g \, dm$,
then a subsequence of the associated PCAFs converges 
in the local uniform topology 
almost surely for quasi-everywhere starting point
\cite[Theorem 4.1]{Nishimori_Tomisaki_Tsuchida_Uemura_2025}.

We focus on the continuity of the Revuz correspondence, that is,
convergence of PCAFs in terms of their associated smooth measures.
The study of convergence of PCAFs is interesting in itself from the viewpoint of the topological properties of the Revuz correspondence, 
but it is also important in applications such as the construction of time-changed processes,
such as Liouville Brownian motions, 
and the discussion of their convergence
\cite{Andres_Kajino_16_Continuity,Croydon_Hambly_Kumagai_17_Time-changes,Garban_Rhodes_Vargas_16_Liouville,Ooi_25_TimeChange}.
We consider standard processes that admit dual processes and possess transition densities.
In particular, the framework contains symmetric Hunt processes
that correspond to regular Dirichlet forms and possess transition densities (cf.\ \cite{Fukushima_Oshima_Takeda_11_Dirichlet}).
In our first result, Theorem \ref{1. thm: the main strong result},
we show that if the $1$-potentials of smooth measures are bounded and converge uniformly, 
then the associated PCAFs converge in the sense that 
the expected supremum of the squared differences of PCAFs over any compact interval converges to zero. 
In our second result, Theorem \ref{1. thm: the main general result},
the first result is extended to more general smooth measures,
and it is shown that 
if the $1$-potentials of smooth measures converge locally uniformly and the process is conservative,
then the associated PCAFs converge uniformly on compacts in probability
(often referred to as the ucp topology).
Our approach differs from that of \cite{Nishimori_Tomisaki_Tsuchida_Uemura_2025}, 
where smooth measures are approximated by a sequence of measures that are absolutely continuous with respect to the reference measure $m$. 
Instead of such an approximation argument, 
we directly estimate the distance between two PCAFs by the distance between the $1$-potentials of the corresponding smooth measures 
(see Theorem \ref{3. thm: key inequality}). 
Using this inequality, our convergence results are proved.

To state our main result,
we introduce several pieces of notation.
The details on our framework are presented in Section~\ref{sec: dual processes}.
We fix a locally compact separable metrizable topological space $S$. 
We let $X = ( (X_{t})_{t \in [0, \infty)}, (P_{x})_{x \in S})$
be a standard process on $S$ admitting a dual process.
We suppose that $X$ admits a transition density $p \colon (0,\infty) \times S \times S \to [0, \infty]$ with respect to a $\sigma$-finite Borel measure $m$.
We then define for each $\alpha \geq 0$ the $\alpha$-potential density $(\resol_{\alpha}(x,y))_{x, y \in S}$ by setting 
\begin{equation}
  \resol_{\alpha}(x,y) 
  \coloneqq 
  \int_{0}^{\infty} e^{-\alpha t} p_{t}(x,y)\, dt.
\end{equation}
For a Borel measure $\mu$ on $S$, a Borel subset $E \subseteq S$ and $\alpha \geq 0$,
we set  
\begin{equation} \label{eq: restriction and potential}
  \mu^{E}(dx) \coloneqq 1_{E}(x)\, \mu(dx),\quad
  \resolop_{\alpha}\mu(x) 
  \coloneqq 
  \int \resol_{\alpha}(x,y)\, \mu(dy).
\end{equation}
We call $\resolop_{\alpha}\mu$ the \emph{$\alpha$-potential} of $\mu$.
Following \cite{Fukushima_Oshima_Takeda_11_Dirichlet},
we define $\smooth_{00}$ to be the collection of finite Borel measures $\mu$ on $S$ such that 
$\mu$ charges no semipolar set and satisfies
\begin{equation}
  \supnorm{\resolop_{1}\mu} 
  \coloneqq 
  \sup\{ |\resolop_{1}\mu(x)| \mid x \in S\} 
  < \infty.
\end{equation}
For Borel measures $\mu, \mu_{1}, \mu_{2}, \ldots$ on $S$,
we consider the following conditions.
Note that, for a Borel measure $\nu$ on $S$ and a Borel measurable function $f$ on $S$,
we write $\inftynorm{f}{\nu}$ for the $L^{\infty}$-norm of $f$ with respect to $\nu$, i.e.,
\begin{equation} \label{eq: infty norm}
  \inftynorm{f}{\nu} 
  \coloneqq 
  \inf\{C > 0 \mid |f(x)| \leq C\ \text{for}\ \nu \text{-a.e.}\ x \in S\},
\end{equation}
where we set $\inf \emptyset \coloneqq \infty$.

\begin{assum} \label{1. assum: stronger assumption}
  The measures $\mu_n$ for $n \ge 1$ and $\mu$ belong to $\smooth_{00}$, and
  \begin{equation}
    \lim_{n \to \infty} 
    \inftynorm{\resolop_{1}\mu_{n} - \resolop_{1}\mu}{\mu_{n} + \mu} 
    = 0.
  \end{equation}
\end{assum}

The following is a generalization of the above condition to measures not necessarily belonging to $\smooth_{00}$.

\begin{assum} \label{1. assum: general assumption} \leavevmode
  There exists an increasing sequence $(V_{k})_{k \geq 1}$ of relatively compact open subsets of $S$ with $\bigcup_{k \geq 1} V_{k} = S$
  such that, for every $k \geq 1$, $\mu^{V_{k}} \in \smooth_{00}$, $\mu_{n}^{V_{k}} \in \smooth_{00}$ for all $n \geq 1$,
  and 
  \begin{equation} \label{1. assum eq: short time control}
    \lim_{n \to \infty} 
    \inftynorm{\resolop_{1}\mu_{n}^{V_{k}} - \resolop_{1}\mu^{V_{k}}}{\mu_{n}^{V_{k}} + \mu^{V_{k}}}
    =0.
  \end{equation}
\end{assum}

Under Assumption~\ref{1. assum: stronger assumption} or \ref{1. assum: general assumption},
it is easy to see from the definition that $\mu_{n}$ and $\mu$ are smooth measures (see Definition \ref{dfn: smooth measure}).
We write $A_{n} = (A_{n}(t))_{t \geq 0}$ (resp.\ $A = (A(t))_{t \geq 0}$) 
for the PCAF associated with $\mu_{n}$ (resp.\ $\mu$) by the Revuz correspondence.
(The details regarding smooth measures and PCAFs are presented in Section~\ref{sec: PCAFs and smooth measures}.)
Our main results are the following.

\begin{thm} \label{1. thm: the main strong result}
  Under Assumption \ref{1. assum: stronger assumption},
  for any $T > 0$,
  \begin{equation}
    \lim_{n \to \infty}
    \sup_{x \in S}
    E_x\!\left[\sup_{0 \leq t \leq T} |A_{n}(t) - A(t)|^{2} \right]
    = 0.
  \end{equation}
\end{thm}

When Assumption \ref{1. assum: general assumption} is satisfied,
by additionally assuming the conservativeness of $X$,
the following weaker convergence of PCAFs is deduced.

\begin{thm} \label{1. thm: the main general result}
  Suppose that Assumption \ref{1. assum: general assumption} is satisfied
  and moreover $X$ is conservative, i.e., 
  \begin{equation}
    P_{x}(X_{t} \in S,\, \forall t \in [0, \infty)) = 1,
    \quad 
    \forall x \in S.
  \end{equation}
  Then, for any $\varepsilon, T > 0$ and $x \in S$,
  \begin{equation} \label{thm eq: convergence of PCAFs in probability}
    \lim_{n \to \infty}
    P_x\!
    \left(
      \sup_{0 \leq t \leq T}|A_{n}(t) - A(t)| > \varepsilon
    \right)
    =0.
  \end{equation}
  As a consequence,
  $A_{n} \xrightarrow{\mathrm{p}} A$ in the local uniform topology under $P_{x}$ for any $x \in S$.
\end{thm}

\begin{rem} \leavevmode
  \begin{enumerate} [label = (\roman*)]
    \item   
      When the $1$-potential density $r_1(x,y)$ is symmetric, i.e., $r_1(x,y)=r_1(y,x)$, and finite for each $(x,y)\in S\times S$,
      the Dirac measure $\delta_x$ at $x$ is a smooth measure for each $x$.
      The corresponding PCAF is the local time $(L(x,t))_{t\geq 0}$ at $x$.
      Theorem~\ref{1. thm: the main strong result} implies that, if $r_1(x,y)$ is jointly continuous in $(x,y)$,
      then the family of local times is jointly continuous in the following sense:
      \begin{equation}
      \lim_{y\to x} \sup_{z\in S} E_z\!\left[\sup_{0\leq t\leq T} |L(x,t)-L(y,t)|^2 \right]=0,
      \qquad
      \forall T>0.
      \end{equation}
      See \cite{Marcus_Rosen_06_Markov,Noda_pre_LocalTimes} for further results on the joint continuity and convergence of local times.
    \item   
      Theorem \ref{1. thm: the main strong result} implies that
      the result of \cite[Theorem 4.1]{Nishimori_Tomisaki_Tsuchida_Uemura_2025} holds for every starting point.
      (Recall that their result holds for quasi-every starting point.)
      Specifically,
      under Assumption \ref{1. assum: stronger assumption},
      there exists a subsequence $(n_{k})_{k \geq 1}$ such that, for all $x \in S$,
      $A_{n_{k}} \to A$ in the local uniform topology almost surely with respect to $P_{x}$.
      This is proven similarly to the fact that convergence in probability implies existence of a subsequence
      that converges almost surely.
    \item     
      In our approach, 
      it is not possible to replace $P_{x}$ of \eqref{thm eq: convergence of PCAFs in probability} with $\sup_{x \in S} P_{x}$.
      This is because the proof depends on 
      the fact that 
      the conservativeness of $X$ ensures that $P_{x}(\tau_{V_{k}} \leq T) \to 0$ as $k \to \infty$ for each $T \in (0, \infty)$ and $x \in S$,
      where $\tau_{V_{k}}$ denotes the exit time of $X$ from $V_{k}$.
      Obviously, this convergence does not hold uniformly with respect to $x \in S$.
      However, if the process is doubly Feller,
      then the convergence holds locally uniformly in $x$;
      see Theorem~\ref{thm: modification for doubly-feller proc}.
    \item    
      In many examples of interest,
      the transition density $p$ is jointly continuous.
      In Section~\ref{sec: sufficient conditions},
      we provide sufficient conditions for Assumptions~\ref{1. assum: stronger assumption} and \ref{1. assum: general assumption}
      that are tractable in applications.
    \end{enumerate}
\end{rem}

The proofs of Theorems \ref{1. thm: the main strong result} and \ref{1. thm: the main general result} are based 
on the following estimate 
for the distance between PCAFs in terms of the distance between the $1$-potentials of the associated smooth measures.
This result itself is new and can be useful for deducing rate of convergence of PCAFs.

\begin{thm} \label{3. thm: key inequality}
  Fix $\mu, \nu \in \smooth_{00}$.
  Let $A$ and $B$ be the associated PCAFs, respectively.
  It then holds that, for any $\alpha, T>0$,
  \begin{align}
    \sup_{x \in S}
    E_{x}\!
    \left[\sup_{0 \leq t \leq T} |A_{t} - B_{t}|^{2} \right]
    & \leq 
    18 (\supnorm{\resolop_{\alpha}\mu} + \|\resolop_{\alpha}\nu\|_{\infty}) \|\resolop_{\alpha} \mu - \resolop_{\alpha}\nu\|_{\infty} \\
    & \quad
    + 
    4e^{2T}(1- e^{-\alpha T})( \supnorm{\resolop_{1}\mu}^{2} + \|\resolop_{1}\nu\|_{\infty}^{2}).
  \end{align}
\end{thm}

The remainder of the article is organized as follows. 
In Section~\ref{sec: Preliminaries}, 
we clarify the framework for our main results and introduce PCAFs and smooth measures.
We also study $\alpha$-potentials of smooth measures.
In Section~\ref{sec: proof of main results},
we prove the main results.
In Section~\ref{sec: sufficient conditions},
we provide sufficient conditions for Assumptions~\ref{1. assum: stronger assumption} and \ref{1. assum: general assumption}.


\section{Preliminaries} \label{sec: Preliminaries}

This section is divided into three subsections.
In Section~\ref{sec: dual processes},
we set out the framework for the arguments of this article.
In Section~\ref{sec: PCAFs and smooth measures},
we introduce PCAFs and smooth measures.
Then, in Section~\ref{sec: potentials}, 
we study properties of $\alpha$-potentials of smooth measures.
Throughout this paper,
we fix a locally compact separable metrizable topological space $S$.
For a subset $E$ of $S$,
we write $\overline{E}$ for the closure of $E$ in $S$.
We define $S_{\Delta} = S \cup \{\Delta\}$ to be the one-point compactification of $S$.
(NB. If $S$ is compact, then we add $\Delta$ to $S$ as an isolated point.)
Any function $f$ defined on $S$ is regarded as a function on $S_{\Delta}$ by setting $f(\Delta) \coloneqq 0$.
We define $\supnorm{f} \coloneqq \sup\{|f(x)| \mid x \in S\}$
and, for a Borel measure $\nu$ on $S$ and a subset $E \subset S$, 
$\inftynorm{f}{\nu}$ to be the $L^{\infty}$-norm of $f$ with respect to $\nu$
(recall this from \eqref{eq: infty norm})
and $\nu^E$ to be the restriction of $\nu$ to $E$.
We use the convention $\inf \emptyset \coloneqq \infty$.
Given a topological space $T$,
we denote the Borel $\sigma$-algebra on $T$ by $\Borel(T)$.
The space $[0, \infty]$ is the usual one-point compactification of $[0, \infty)$.


\subsection{Standard processes admitting dual processes} \label{sec: dual processes}

In this subsection,
we clarify the framework
that is assumed in the rest of this paper.
We follow the setting of \cite{Revuz_70_Mesures}.

Let 
\begin{equation}
  X = (\Omega, \sigalg, (X_{t})_{t \in [0, \infty]}, (P_{x})_{x \in S_{\Delta}}, (\theta_{t})_{t \in [0,\infty]})
\end{equation}
be a standard process on $S$
(see \cite[Definition~A.1.23(i)]{Chen_Fukushima_12_Symmetric}).
Here, $(\Omega, \sigalg)$ denotes the measurable space, and $\theta_{t}$ denotes the shift operator,
i.e., a map $\theta_{t}\colon \Omega \to \Omega$ satisfying $X_{s} \circ \theta_{t} = X_{s+t}$ for any $s \in [0, \infty]$.
Note that $X_{\infty}(\omega) = \Delta$ for all $\omega \in \Omega$.
We write $\filt_{*} = (\filt_{t})_{t  \in [0,\infty]}$ 
for the minimum augmented admissible filtration (see \cite[p.\ 397]{Chen_Fukushima_12_Symmetric})
and $\zeta \coloneqq \inf\{t \in [0, \infty] \mid X_{t} = \Delta\}$
for the lifetime of $X$.
Throughout this paper, we assume the following duality and absolute continuity condition.

\begin{assum} \label{assum: heat kernel}
  There exists a standard process
  \begin{equation}
    \hat{X} = (\hat{\Omega},\hat{\sigalg}, (\hat{X}_{t})_{t \in [0, \infty]}, 
      (\hat{P}_{x})_{x \in S_{\Delta}}, (\hat{\theta}_{t})_{t \in [0,\infty]}),
  \end{equation} 
  a $\sigma$-finite Borel measure $m$ on $S$,
  and a Borel measurable function $p \colon (0, \infty) \times S \times S \to [0,\infty]$,
  called the \emph{heat kernel} of $X$, satisfying the following.
  \begin{enumerate} [label = \textup{(DAC\arabic*)}, leftmargin = *]
    \item For all $t > 0$ and $x \in S$, $P_x(X_t \in dy) = p_t(x, y)\, m(dy)$.
    \item For all $t > 0$ and $y \in S$, $\hat{P}_y(\hat{X}_t \in dx) = p_t(x, y)\, m(dx)$.
    \item For any $s, t > 0$ and $x, y \in S$, 
      \begin{equation} \label{2. eq: C-K equation}
        p_{t+s}(x, y) = \int_{S} p_{t}(x, z) p_{s}(z, y)\, m(dz).
      \end{equation}
  \end{enumerate}
\end{assum}

A sufficient condition for Assumption~\ref{assum: heat kernel} can be found in \cite[Theorem~1]{Yan_88_A_formula}.
We note that, by \cite[Corollary~1.12 and Remark~1.13 in Chapter~VI]{Blumenthal_Getoor_68_Markov}, 
the measure $m$ is an excessive reference measure of $X$
(see \cite[Definitions~1.1 and~1.10 in Chapters~V and~VI, respectively]{Blumenthal_Getoor_68_Markov} for these definitions).
Under Assumption~\ref{assum: heat kernel}, for each $\alpha \geq 0$, we define the \emph{$\alpha$-potential density} by
\begin{equation}
  \resol_{\alpha}(x,y) 
  \coloneqq 
  \int_{0}^{\infty} e^{-\alpha t} p_{t}(x,y)\, dt,
  \qquad 
  x,y \in S.
\end{equation}

\begin{rem}
  The class of processes satisfying Assumption~\ref{assum: heat kernel} is sufficiently large.
  For example,
  if $X$ is a Hunt process associated with a regular Dirichlet form (cf.\ \cite{Fukushima_Oshima_Takeda_11_Dirichlet})
  admitting a transition density,
  then it satisfies Assumption~\ref{assum: heat kernel}.
  (Indeed, one can choose $\hat{X}$ to be $X$ itself.)
\end{rem}

By using \eqref{2. eq: C-K equation},
it is easy to see that, for any $x, y \in S$ and $\alpha \geq \beta > 0$, 
  \begin{align} 
    \resol_{\beta}(x,y)
    &= 
    \resol_{\alpha}(x,y) + (\alpha - \beta) \int_{S} \resol_{\beta}(x, z) \resol_{\alpha}(z, y)\, m(dz) \\
    &= 
    \resol_{\alpha}(x,y) + (\alpha - \beta) \int_{S} \resol_{\alpha}(x, z) \resol_{\beta}(z, y)\, m(dz)
    \label{2. eq: resolvent identity}
  \end{align}
(cf.\ \cite[Lemma 3.3.4]{Marcus_Rosen_06_Markov}).
For every Borel subset $E$ of $S$,
we write $\sigma_E$ and $\tau_{E}$ for the first hitting time and exit time of $E$ by $X$, i.e., 
\begin{equation} \label{2. eq: hitting time}
  \sigma_{E} 
  \coloneqq 
  \inf\{t \in (0, \infty) \mid X_{t} \in E\},
\end{equation}
\begin{equation} \label{2. eq: exit time}
  \tau_{E} 
  \coloneqq 
  \inf\{t \in [0, \infty) \mid X_{t} \notin E\}.
\end{equation}


\subsection{PCAFs and smooth measures} \label{sec: PCAFs and smooth measures}

In this subsection,
we introduce the main objects: PCAFs and smooth measures.
We proceed with the same setting as in Section~\ref{sec: dual processes}.

\begin{dfn} [{PCAF}] \label{dfn: PCAF}
  Let $A = (A_{t})_{t \geq 0}$ be an $\filt_{*}$-adapted non-negative stochastic process. 
  It is called a \textit{positive continuous additive functional (PCAF)} of $X$ if 
  there exists a set $\Lambda \in \filt_{\infty}$, called a \textit{defining set} of $A$, satisfying the following.
  \begin{enumerate} [label = (\roman*)]
    \item It holds that $P_{x}(\Lambda) = 1$ for all $x \in S$ and $\theta_t(\Lambda) \subseteq \Lambda$ for all $t \geq 0$.
    \item 
      For every $\omega \in \Lambda$,
      $A_{0}(\omega) = 0$,
      the function $t \mapsto A_{t}(\omega)$ from $[0,\infty)$ to $[0, \infty]$ is continuous,
      $A_{t}(\omega) < \infty$ for all $t < \zeta(\omega)$, 
      $A_{t}(\omega) = A_{\zeta(\omega)}(\omega)$ for all $t \geq \zeta(\omega)$,
      and $A_{t+s}(\omega) = A_{t}(\omega) + A_{s}(\theta_{t}\omega)$
      for all $s, t \geq 0$.
  \end{enumerate} 
  We set $A_{\infty} \coloneqq \lim_{t \to \infty} A_{t}$ on $\Lambda$ 
  and $A_{\infty} \coloneqq 0$ otherwise.
\end{dfn}

\begin{rem}
  The above definition follows that of PCAFs \emph{in the strict sense} in \cite[p.\ 222]{Fukushima_Oshima_Takeda_11_Dirichlet}.
  Revuz's definition of PCAF \cite[Section~I.3]{Revuz_70_Mesures} is slightly weaker, and the existence of a defining set is not required. 
  However, it is not difficult to find a version of his PCAF which is a PCAF in the strict sense.
  Indeed, by \cite[Theorem~2.1 in Chapter~V]{Blumenthal_Getoor_68_Markov}, any PCAF is equivalent to a so-called perfect PCAF
  (see \cite[Definition~1.3 in Chapter~IV]{Blumenthal_Getoor_68_Markov} for its definition).
  In that theorem, such a perfect PCAF is constructed from PCAFs constructed in \cite[Theorem~3.16 in Chapter~IV]{Blumenthal_Getoor_68_Markov},
  which can be easily verified to be PCAFs in the strict sense.
  Thus, one can deduce that the perfect PCAF constructed in \cite[Theorem~2.1 in Chapter~V]{Blumenthal_Getoor_68_Markov} is a PCAF in the strict sense.
\end{rem}

For the definition of smooth measures,
recall the notion of semipolar sets from \cite[Definition~3.1 in Chapter~II]{Blumenthal_Getoor_68_Markov}.
Also, recall from \eqref{eq: restriction and potential} that,
for a Borel measure $\mu$ on $S$,
the $\alpha$-potential $\resolop_{\alpha}\mu \colon S \to [0, \infty]$ of $\mu$ is defined by 
\begin{equation} \label{2. eq: alpha-potential}
  \resolop_{\alpha}\mu(x) 
  \coloneqq 
  \int_{S} \resol_{\alpha}(x, y)\, \mu(dy).
\end{equation}

\begin{dfn} [{Smooth measure, \cite[Theorem~VI.1]{Revuz_70_Mesures}}] \label{dfn: smooth measure}
  A Borel measure $\mu$ on $S$ is said to be a \textit{smooth measure} if it satisfies the following conditions:
  \begin{enumerate} [label = (\roman*)]
    \item \label{dfn item: charging no set of zero capacity}
      $\mu$ charges no semipolar sets, i.e., 
      $\mu(E) = 0$ for any semipolar $E \in \Borel(S)$;
    \item \label{dfn item: existence of nest}
      there exists an increasing sequence $(E_{n})_{n \geq 1}$ of Borel subsets 
      such that, $S = \bigcup_{n \geq 1} E_n$, $\mu(E_{n}) < \infty$ for each $n$, 
      \begin{equation}
        \|\resolop_1 \mu^{E_n}\|_\infty
        =
        \sup_{x \in S} \int_{E_n} \resol_1(x,y)\, \mu(dy) < \infty, \quad \forall n \geq 1,
      \end{equation}
      and 
      $P_x( \lim_{n \to \infty} \tau_{E_n} \geq \zeta) = 1$ for all $x \in S$.
  \end{enumerate}
  We write $\smooth_{1}$ for the collection of smooth measures,
  and write $\smooth_{00}$ (resp.\ $\smooth_{00}^{(0)}$) 
  for the collection of finite Borel measures $\mu$ on $S$ 
  such that $\mu$ charges no semipolar sets and satisfies $\supnorm{\resolop_{1}\mu} < \infty$ (resp.\ $\supnorm{\resolop_{0}\mu} < \infty$).
  (NB.\ It holds that $\smooth_{00}^{(0)} \subseteq \smooth_{00} \subseteq \smooth_{1}$.)
\end{dfn}

\begin{rem}
 The above definition of smooth measures coincides with that of smooth measures \emph{in the strict sense} in the Dirichlet form setting (see  \cite[p.~238]{Fukushima_Oshima_Takeda_11_Dirichlet}).
\end{rem}

It is known that there is a one-to-one correspondence between PCAFs and smooth measures,
which is often referred to as the \textit{Revuz correspondence} due to Revuz \cite{Revuz_70_Mesures}.
Below, note that two PCAFs $A$ and $B$ are said to be equivalent 
if there exists a set $\Lambda \in \filt_\infty$ which is a defining set of both $A$ and $B$,
satisfies $P_x(\Lambda) = 1$ for all $x \in S$, and $A_t(\omega) = B_t(\omega)$ for all $(t, \omega) \in [0,\infty) \times \Lambda$.

\begin{thm} [{\cite{Revuz_70_Mesures}; see also \cite[Section~4]{Kajino_Noda_pre_Generalized}}] \label{thm: Revuz correspond}
  There is a one-to-one correspondence between equivalence classes of PCAFs $(A_{t})_{t \geq 0}$
  and smooth measures $\mu$ characterized by the following relation.
  \begin{enumerate} [label = \textup{(RC)}, leftmargin = *]
    \item \label{2. item: Revuz correspondence}
      For any $\alpha > 0$ and non-negative Borel measurable function $f$ on $S$,
      \begin{equation} \label{2. eq: Revuz correspondence}
        E_x\!\left[
          \int_{0}^{\infty} e^{-\alpha t} f(X_t)\, dA_t
        \right]
        = 
        \int_{S} \resol_{\alpha}(x,y) f(y)\, \mu(dy),
        \quad 
        \forall x \in S.
      \end{equation}
  \end{enumerate}
\end{thm}

For a PCAF $A$,
we refer to the unique Borel measure $\mu \in \smooth_{1}$ satisfying \ref{2. item: Revuz correspondence} as the measure associated with $A$.
Similarly,
for $\mu \in \smooth_{1}$,
we refer to a PCAF $A$ satisfying \ref{2. item: Revuz correspondence},
which is unique up to the equivalence relation,
as the PCAF associated with $\mu$.
The following result will be used in the proof of our main results.

\begin{lem} [{\cite[Theorem~A.3.5(iii)]{Chen_Fukushima_12_Symmetric}}]\label{2. lem: restriction of smooth measures}
  Fix $\mu \in \smooth_{1}$ and a Borel subset $E \subseteq S$.
  Let $A=(A_{t})_{t \geq 0}$ be the associated PCAF.
  Then $(B_{t})_{t \geq 0}$ defined by setting $B_{t} \coloneqq \int_{0}^{t} 1_{E}(X_{s})\, dA_{s}$ is a PCAF,
  and the associated smooth measure is $\mu^{E}$.  
\end{lem}


\subsection{Properties for \texorpdfstring{$\alpha$-potentials}{alpha-potentials}} \label{sec: potentials}

In this subsection, we study properties of $\alpha$-potentials introduced in \eqref{2. eq: alpha-potential}.
The main results are domination principles for $\alpha$-potentials presented in Propositions~\ref{prop: domination principle} and~\ref{2. prop: maximum principle for difference of potentials}.
We also provide a useful sufficient condition for a Radon measure to be smooth in Proposition~\ref{2. prop: smoothness condition}.
We proceed with the same setting as in Section~\ref{sec: dual processes}.

\begin{lem} \label{lem: resolvent identity for potentials}
  Fix a Borel measure $\mu$ on $S$.
  For any $\alpha \geq \beta > 0$, we have that 
  \begin{align} 
    \resolop_{\beta}\mu(x) 
    &=
    \resolop_{\alpha}\mu(x) 
    +
    (\alpha -\beta) 
    \int_{S} \resol_{\beta}(x, z)\, \resolop_{\alpha}\mu(z)\, m(dz)\\
    &=
    \resolop_{\alpha}\mu(x) 
    +
    (\alpha -\beta) 
    \int_{S} \resol_{\alpha}(x,z)\, \resolop_{\beta}\mu(z)\, m(dz)
    \label{2. lem eq: resolvent identity for potentials}
  \end{align}
  In particular, if $\mu \in \smooth_{00}$,
  then the above equations hold for any $\alpha, \beta > 0$.
  As a consequence of \eqref{2. lem eq: resolvent identity for potentials},
  if $\supnorm{\resolop_{\alpha}\mu} < \infty$ for some $\alpha > 0$,
  then $\supnorm{\resolop_{\beta}\mu} < \infty$ for any $\beta > 0$.
\end{lem}

\begin{proof}
  Equation \eqref{2. lem eq: resolvent identity for potentials} follows from \eqref{2. eq: resolvent identity}.
  If $\mu \in \smooth_{00}$,
  then the terms are all finite and so we can consider subtraction,
  which yields that \eqref{2. lem eq: resolvent identity for potentials} holds for any $\alpha, \beta > 0$.
  Noting that 
  \begin{equation}
    \int_{S} \resol_{\beta}(x, z)\, m(dz) 
    = 
    \int_{0}^{\infty} \int_{S} e^{-\beta t} p_{t}( x, z)\, m(dz)\, dt  
    \leq 
    \int_{0}^{\infty} e^{-\beta t}\, dt   
    = \beta^{-1},
  \end{equation}
  we obtain the last assertion of the result.
\end{proof}

Recall that a non-negative Borel measurable function $f$ on $S$ is said to be \emph{$\alpha$-excessive}
if $e^{-\alpha t} E_{x}[f(X_{t})] \leq f(x)$ for all $t \geq 0$ and $x \in S$
and $e^{-\alpha t} E_{x}[f(X_{t})] \uparrow f(x)$ as $t \downarrow 0$ for all $x \in S$.
Below, we collect some fundamental results about excessive functions.

\begin{lem} \label{2. lem: basics about excessiveness}
  Fix $\alpha > 0$.
  The following statements hold.
  \begin{enumerate} [label = \textup{(\roman*)}]
    \item \label{2. lem item: constant is excessive}
      Any non-negative constant function is $\alpha$-excessive.
    \item \label{2. lem item: operation on quasi-excessive functions}
      If $f$ and $g$ are $\alpha$-excessive, then so is $f + g$.
    \item \label{2. lem item: potential is excessive}
      For any Borel measure $\mu$, $\resolop_{\alpha}\mu$ is $\alpha$-excessive.
  \end{enumerate}
\end{lem}

\begin{proof}
  See \cite[Proposition~2.2 in Chapter~II]{Blumenthal_Getoor_68_Markov}
  for the first two assertions. 
  As for the last assertion,
  by using \eqref{2. eq: C-K equation},
  we deduce that 
  \begin{equation}
    e^{-\alpha t} E_{x}[\resolop_{\alpha}\mu (X_{t})]
    = 
    \int_{S}
    \int_{t}^{\infty} e^{-\alpha s} p(s, x, z)\, ds \mu(dz),
  \end{equation}
  which yields that $\resolop_{\alpha}\mu$ is $\alpha$-excessive.
\end{proof}

We are now ready to prove domination principles for $\alpha$-potentials.

\begin{prop} \label{prop: domination principle}
  Fix $\alpha > 0$.
  Let $\mu$ be a smooth measure and $f$ be an $\alpha$-excessive function.
  If $\resolop_\alpha \mu \leq f$, $\mu$-a.e., then $\resolop_\alpha \mu (x) \leq f(x)$ for all $x \in S$.
\end{prop}

\begin{proof}
  When $\mu$ is absolutely continuous with respect to the reference measure $m$,
  the desired result can be found in \cite[Theorem~3 in Section~3.4]{Chung_Walsh_05_Markov}.
  Using the Revuz correspondence \eqref{2. eq: Revuz correspondence},
  the proof of \cite[Theorem~3 in Section~3.4]{Chung_Walsh_05_Markov}
  can be adapted to the present setting.
  Fix $x \in S$.
  Define 
  \begin{equation}
    B \coloneqq \{y \in S \mid \resolop_\alpha \mu (y) \leq f(y) \}.
  \end{equation}
  Since $\resolop_\alpha \mu \leq f$, $\mu$-a.e.,
  we have $\mu(dy) = 1_B(y)\, \mu(dy)$.
  From Lemma~\ref{2. lem: restriction of smooth measures},
  it follows that
  \begin{equation}
    A_t = \int_0^t 1_B(X_s)\, dA_s 
    \quad 
    \text{for all $t \geq 0$, $P_x$-a.s.}
  \end{equation}
  Using this and \eqref{2. eq: Revuz correspondence}, we deduce that 
  \begin{equation}
    \resolop_\alpha \mu(x) 
    = 
    E_x\!\left[\int_0^\infty e^{-\alpha t} dA_t\right]
    = 
    E_x\!\left[\int_0^\infty e^{-\alpha t} 1_B(X_t)\, dA_t\right]
    =
    E_x\!\left[\int_{\sigma_B}^\infty e^{-\alpha t} dA_t\right].
  \end{equation}
  By \cite[Theorem~10.19 in Chapter~I]{Blumenthal_Getoor_68_Markov},
  there exists an increasing sequence $(K_n)_{n \geq 1}$ of compact subsets 
  such that $K_n \subseteq B$ and $\sigma_{K_n} \downarrow \sigma_B$, $P_x$-a.s.
  The monotone convergence theorem and the strong Markov property yield that 
  \begin{equation} 
    E_x\!\left[\int_{\sigma_B}^\infty e^{-\alpha t} dA_t\right]
    =
    \lim_{n \to \infty}
    E_x\!\left[\int_{\sigma_{K_n}}^\infty e^{-\alpha t} dA_t\right]
    = 
    \lim_{n \to \infty}
    E_x\!\left[ 1_{\{\sigma_{K_n} < \infty\}}\, e^{-\alpha \sigma_{K_n}} \resolop_\alpha \mu(X_{\sigma_{K_n}})\right].
  \end{equation}
  On the event $\{\sigma_{K_n} < \infty\}$,
  since $X_{\sigma_{K_n}} \in K_n \subseteq B$,
  we have $\resolop_\alpha \mu(X_{\sigma_{K_n}}) \leq f(X_{\sigma_{K_n}})$.
  Thus, we obtain that 
  \begin{equation} \label{pf: 1. domination principle}
    \resolop_\alpha \mu(x) 
    \leq    
    \liminf_{n \to \infty}
    E_x\!\left[ 1_{\{\sigma_{K_n} < \infty\}}\, e^{-\alpha \sigma_{K_n}} f(X_{\sigma_{K_n}})\right].
  \end{equation}
  Set $Y_t \coloneqq e^{-\alpha t} f(X_t)$ for all $t \geq 0$.
  Then $(Y_t)_{t \geq 0}$ is a non-negative $(\filt_t)_{t \geq 0}$-supermartingale
  (see \cite[Proposition~1]{Chung_Walsh_05_Markov}).
  Write $Y_\infty$ for the almost-sure limit of $Y_t$ as $t \to \infty$,
  guaranteed by the martingale convergence theorem.
  Since $f$ is finely continuous,
  $(Y_t)_{t \geq 0}$ is right-continuous
  (cf.\ \cite[Proposition~4.2 and~Theorem~4.8 in Chapter~II]{Blumenthal_Getoor_68_Markov}).
  Thus, we can apply the optional stopping theorem (cf.\ \cite[Theorem~4 in Section~1.4]{Chung_Walsh_05_Markov})
  to obtain that 
  \begin{equation} \label{pf: 2. domination principle}
    f(x) 
    = E_x[Y_0] 
    \geq E_x[Y_{\sigma_{K_n}}] 
    \geq E_x\!\left[1_{\{\sigma_{K_n} < \infty\}}\, e^{-\alpha \sigma_{K_n}} f(X_{\sigma_{K_n}})\right]
  \end{equation}
  for all $n$.
  Now the desired result is immediate from \eqref{pf: 1. domination principle} and \eqref{pf: 2. domination principle}.
\end{proof}

The following is a domination principle on differences of $\alpha$-potentials.

\begin{prop} \label{2. prop: maximum principle for difference of potentials}
  Fix $\alpha > 0$.
  For any $\mu, \nu \in \smooth_{00}$,
  it holds that 
  \begin{equation}
    \supnorm{\resolop_{\alpha}\mu - \resolop_{\alpha}\nu}
     = 
     \inftynorm{\resolop_{\alpha}\mu - \resolop_{\alpha}\nu}{\mu + \nu}.
  \end{equation}
\end{prop}

\begin{proof}
  Fix $\mu, \nu \in \smooth_{00}$ and set $\varepsilon \coloneqq \inftynorm{\resolop_{\alpha}\mu - \resolop_{\alpha}\nu}{\mu + \nu}$.
  We then have that 
  \begin{equation}
    \resolop_{\alpha}\mu 
    \leq
    \varepsilon
    + 
    \resolop_{\alpha}\nu,
    \quad 
    \mu\text{-a.e.}  
  \end{equation}
  Since the right-hand side is $\alpha$-excessive,
  Proposition~\ref{prop: domination principle} implies that the above inequality holds everywhere.
  It follows that 
  \begin{equation} \label{eq: one-side estimate on sup norm}
    \resolop_\alpha \mu(x) - \resolop_\alpha \nu(x) \leq \varepsilon,
    \qquad 
    \forall x \in S.
  \end{equation}
  By the definition of $\varepsilon$,
  we also have that 
  $
    \resolop_{\alpha} \nu      
    \leq 
    \varepsilon
    + 
    \resolop_{\alpha}\mu
  $, $\nu$-a.e.
  So, by the same argument, we deduce that 
  \begin{equation} \label{eq: other-side estimate on sup norm}
    \resolop_{\alpha}\nu(x) - \resolop_{\alpha}\mu(x) \leq \varepsilon,
    \quad 
    \forall x \in S.
  \end{equation}
  From \eqref{eq: one-side estimate on sup norm} and \eqref{eq: other-side estimate on sup norm},
  it follows that $\supnorm{\resolop_{\alpha}\mu - \resolop_{\alpha}\nu} \leq \varepsilon$.
  The converse inequality is obvious and so we obtain the desired result.
\end{proof}

We next provide a useful sufficient condition for a Radon measure to be smooth, in Proposition~\ref{2. prop: smoothness condition} below.
To this end, we first present a simple consequence of the strong Markov property.

\begin{prop} \label{prop: balayage}
  Let $D$ be an open subset of $S$.
  It then holds that, for each $\alpha \geq 0$ and $x \in S$, 
  \begin{equation}
    \resol_\alpha(x,y) = E_x\!\left[ e^{-\alpha \sigma_D} \resol_\alpha(X_{\sigma_D}, y)\, 1_{\{\sigma_D < \infty\}} \right],
    \quad \forall y \in D.
  \end{equation}
  In particular, for each $x \in S$ and $y \in D$, 
  it holds that, for Lebesgue-almost every $t > 0$,
  \begin{equation}
    p_t(x,y) = E_x\!\left[ p_{t - \sigma_D}(X_{\sigma_D}, y)\, 1_{\{\sigma_D < t\}} \right].
  \end{equation}
\end{prop}

\begin{proof}
  Fix $x \in S$.
  For any Borel subset $B \subseteq D$, we deduce from the strong Markov property that 
  \begin{align}
    \int_B \resol_\alpha(x,y)\, m(dy) 
    &= 
    E_x\!\left[ \int_0^\infty e^{-\alpha t} 1_{B}(X_t)\, dt \right]\\
    &= 
    E_x\!\left[ \int_{\sigma_D}^\infty e^{-\alpha t} 1_{B}(X_t)\, dt \cdot 1_{\{\sigma_D < \infty\}} \right]\\
    &= 
    E_x\!\left[ e^{-\alpha \sigma_D} E_{X_{\sigma_D}}\!\left[\int_0^\infty e^{-\alpha t} 1_B(X_t)\, dt \right] \cdot 1_{\{\sigma_D < \infty\}} \right]\\
    &= 
    E_x\!\left[ e^{-\alpha \sigma_D} ( \resolop_\alpha 1_B ) (X_{\sigma_D}) \cdot 1_{\{\sigma_D < \infty\}} \right]\\
    &=
    \int_B E_x\!\left[ e^{-\alpha \sigma_D} \resol_\alpha (X_{\sigma_D}, y) \cdot 1_{\{\sigma_D < \infty\}} \right] m(dy).
  \end{align}
  Thus, 
  \begin{equation}
    \resol_\alpha(x,y) = E_x\!\left[ e^{-\alpha \sigma_D} \resol_\alpha (X_{\sigma_D}, y) \cdot 1_{\{\sigma_D < \infty\}} \right],
    \quad m\text{-a.e.\ on } D.
  \end{equation}
  Since the functions on both sides are $\alpha$-coexcessive with respect to $y$
  (i.e., $\alpha$-excessive for the dual process $\hat{X}$), 
  the above equality holds for every $y \in D$
  (cf.\ \cite[Proof of Theorem~1.16 in Chapter~VI]{Blumenthal_Getoor_68_Markov}).
  Observe that 
  \begin{align}
    E_x\!\left[ e^{-\alpha \sigma_D} \resol_\alpha (X_{\sigma_D}, y) \cdot 1_{\{\sigma_D < \infty\}} \right]
    &=
    E_x\!\left[ e^{-\alpha \sigma_D} \int_0^\infty e^{-\alpha t} p_t(X_{\sigma_D}, y)\, dt \cdot 1_{\{\sigma_D < \infty\}} \right]\\
    &=
    E_x\!\left[ \int_{\sigma_D}^\infty e^{-\alpha t} p_{t - \sigma_D} (X_{\sigma_D}, y)\, dt \cdot 1_{\{\sigma_D < \infty\}} \right]\\
    &= 
    \int_0^\infty e^{-\alpha t} E_x\!\left[ p_{t - \sigma_D} (X_{\sigma_D}, y)\, 1_{\{\sigma_D < t\}} \right] dt.
  \end{align}
  Thus, the last assertion follows from the uniqueness of the Laplace transform.
\end{proof}

\begin{prop} \label{2. prop: smoothness condition}
  Let $\mu$ be a Radon measure on $S$.
  If, for any compact subset $K$ of $S$, 
  \begin{equation} \label{2. prop item eq: 2. smoothness condition}
    \lim_{\delta \to 0}
    \sup_{x \in K}
    \int_{0}^{\delta} \int_{K} p_{t}(x, y)\, dt\, \mu(dy)
    < \infty,
  \end{equation}
  Then $\|\resolop_1 \mu^D\|_\infty < \infty$ for any relatively compact Borel subset $D$ of $S$. 
  In particular, 
  if \eqref{2. prop item eq: 2. smoothness condition} holds for each compact subset $K$ 
  and $\mu$ charges no semipolar sets,
  then $\mu \in \smooth_{1}$.
\end{prop}

\begin{proof}
  Fix a relatively compact open subset $D$ of $S$.
  By Proposition~\ref{prop: balayage},
  \begin{align}
    \int_0^\delta \int_D p_t(x,y)\, \mu(dy)\, dt 
    &= 
    E_x\!\left[ \int_0^\delta \int_D p_{t - \sigma_D}(X_{\sigma_D}, y) \cdot 1_{\{\sigma_D < t\}}\, \mu(dy)\, dt \right]\\
    &=
    E_x\!\left[ \int_0^{\delta - \sigma_D} \int_D p_t(X_{\sigma_D}, y)\, \mu(dy)\, dt \right]\\
    &\leq 
    \sup_{z \in \overline{D}}
    \int_0^\delta \int_D p_t(z, y)\, \mu(dy)\, dt.
  \end{align}
  Thus,
  \begin{equation} \label{pr: Maximal principle for potential}
    \sup_{x \in S} \int_0^\delta \int_D p_t(x,y)\, \mu(dy)\, dt  
    \leq    
    \sup_{z \in \overline{D}}
    \int_0^\delta \int_D p_t(z, y)\, \mu(dy)\, dt.
  \end{equation}
  By \eqref{2. prop item eq: 2. smoothness condition},
  we can find a small $\delta$ such that the right-hand side in the above inequality is finite.
  Hence, if we fix such a $\delta > 0$, then 
  \begin{equation}
    M \coloneqq \sup_{x \in S} \int_0^\delta \int_D p_t(x,y)\, \mu(dy)\, dt < \infty.
  \end{equation}
  We are now able to follow the argument in the proof of \cite[Proposition 2.7]{Mori_21_Kato_Sobolev}.
  Namely, by using the Chapman--Kolmogorov equation~\eqref{2. eq: C-K equation},
  we deduce that 
  \begin{align}
    \resolop_1\mu^D(x) 
    &= 
    \int_D \int_0^\infty e^{-t} p_t(x,y)\, dt\, \mu(dy)\\
    &=
    \sum_{k=0}^\infty \int_D \int_{\delta k}^{\delta(k+1)} e^{-t} p_t(x,y)\, dt\, \mu(dy)\\
    &=
    \sum_{k=0}^\infty e^{-\delta k} \int_D \int_0^\delta e^{-s} p_{s + \delta k}(x,y)\, ds\, \mu(dy)\\
    &= 
    \sum_{k=0}^\infty e^{-\delta k} \int_D \int_0^\delta e^{-s} \int_S p_{\delta k}(x, z) p_s(z,y)\, m(dz)\, ds\, \mu(dy)\\
    &\leq
    M \sum_{k=0}^\infty e^{-\delta k} \int_S p_{\delta k}(x, z)\, m(dz)\\
    &\leq    
    \frac{M}{1-e^{-\delta}},
    \label{pr: potential bound by CK equation}
  \end{align}
  which implies that $\|\resolop_{1}\mu^{D}\|_{\infty} < \infty$.
  By the local compactness of $S$, 
  any relatively compact Borel subset $D'$ is contained in some relatively compact open subset $D''$.
  Since $\resolop_{1}\mu^{D'} \leq \resolop_{1}\mu^{D''}$,
  we obtain that $\supnorm{\resolop_{1}\mu^{D'}} < \infty$,
  which proves the first assertion.
  Now, assume that $\mu$ charges no semipolar set.
  Let $(D_{k})_{k \geq 1}$ be an increasing sequence of relatively compact open subsets of $S$ with $\bigcup_{k \geq 1} D_{k} = S$.
  Since $\mu^{D_{k}} \in \smooth_{00}$ and $\tau_{D_k} \to \zeta$ almost surely under $P_{x}$ for every $x \in S$,
  we deduce that $\mu \in \smooth_{1}$.
\end{proof}


\section{Proof of main results} \label{sec: proof of main results}

In this section,
we prove the main results.
We first show Theorem~\ref{3. thm: key inequality}.
Then, using it, we establish the convergence results, 
Theorems \ref{1. thm: the main strong result} and \ref{1. thm: the main general result}.
We proceed in the same setting as in Section~\ref{sec: dual processes}.


\subsection{Proof of Theorem~\ref{3. thm: key inequality}} \label{sec: The key estimate on difference of PCAF}

In this subsection, we prove Theorem~\ref{3. thm: key inequality}.
We will use the following moment formula for the product of PCAFs.

\begin{lem} [{\cite[Corollary~2.10(iii)]{Kajino_Noda_pre_Generalized}}] \label{2. lem: Kac's formula}
  Fix $\mu, \nu \in \smooth_{1}$ and write $A, B$ for the associated PCAFs.
  Then, for any $x \in S$,
  \begin{equation}
    E_{x}[A_{\infty} B_{\infty}]
    =
    \int_{S} \resol_{0}(x,y) \resolop_{0}\mu(y)\, \nu(dy)
    +
    \int_{S} \resol_{0}(x,y) \resolop_{0}\nu(y)\, \mu(dy).
  \end{equation}
\end{lem}

We first show an estimate for PCAFs associated with measures in $\smooth_{00}^{(0)}$,
where we recall the class $\smooth_{00}^{(0)}$ from Definition~\ref{dfn: smooth measure}.

\begin{lem} \label{lem: L^2 inequality for 0-order measure}
  Fix $\mu, \nu \in \smooth_{00}^{(0)}$.
  Let $A$ and $B$ be the associated PCAFs, respectively.
  Then it holds that
  \begin{equation}
    \sup_{x \in S}
    E_{x}
    \left[ \sup_{0 \leq t \leq \infty} |A_{t} - B_{t}|^{2} \right]
    \leq
    18 (\|\resolop_{0}\mu\|_{\infty} + \|\resolop_{0}\nu\|_{\infty}) \|\resolop_{0}\mu - \resolop_{0}\nu\|_{\infty} .
  \end{equation}
\end{lem}

\begin{proof}
  Fix $x \in S$.
  We note that, by Lemma~\ref{2. lem: Kac's formula}, $A_{\infty}$ and $B_{\infty}$ are square-integrable with respect to $P_{x}$.
  Define a $\filt_{*}$-martingale $(M_{t})_{t \in [0, \infty]}$ by setting $M_{t} \coloneqq E_{x}[A_{\infty} - B_{\infty}\, |\, \filt_{t}]$.
  Doob's martingale inequality yields that 
  \begin{equation}
    E_x\!\left[ \sup_{0 \leq t \leq \infty} M_{t}^{2} \right]
    \leq 
    4E_{x}[M_{\infty}^{2}]
    = 
    4E_x\!\left[ |A_{\infty} - B_{\infty}|^{2} \right].
  \end{equation}
  Using Lemma~\ref{2. lem: Kac's formula},
  we obtain that 
  \begin{align}
    E_x\!\left[ |A_{\infty} - B_{\infty}|^{2} \right] 
    &= 
    E_{x}[A_{\infty}^{2}] + E_{x}[B_{\infty}^{2}] - 2E_{x}[A_{\infty}B_{\infty}] \\
    &= 
    2 \int_{S} \resol_{0}(x,y) \resolop_{0}\mu(y)\, \mu(dy) 
    + 
    2 \int_{S} \resol_{0}(x,y) \resolop_{0}\nu(y)\, \nu(dy) \\
    &\qquad
    - 
    2\int_{S} \resol_{0}(x,y) \resolop_{0}\mu(y)\, \nu(dy)
    -
    2\int_{S} \resol_{0}(x,y) \resolop_{0}\nu(y)\, \mu(dy) \\ 
    &\leq    
    2 (\|\resolop_{0}\mu\|_{\infty} + \|\resolop_{0}\nu\|_{\infty}) \|\resolop_{0}\mu - \resolop_{0}\nu\|_{\infty}.
  \end{align}
  Since we have that $A_{\infty} - B_{\infty} = (A_{\infty} - B_{\infty}) \circ \theta_{t} + (A_{t} - B_{t})$ for every $t \geq 0$,
  the Markov property yields that 
  \begin{equation}
    M_{t}
    = 
    A_{t} - B_{t}
    + 
    E_{X_{t}}[A_{\infty} - B_{\infty}]
    = 
    A_{t} - B_{t}
    + 
    \resolop_{0}\mu(X_{t}) - \resolop_{0}\nu(X_{t}),
    \quad
    \forall t \geq 0.
  \end{equation}
  Thus, we deduce that 
  \begin{align}
    \lefteqn{E_x\!\left[ \sup_{0 \leq t \leq \infty} |A_{t} - B_{t}|^{2} \right]}\\
    & \leq
    2 E_x\!\left[ \sup_{0 \leq t \leq \infty} M_{t}^{2} \right]
    + 
    2 E_x\!\left[ \sup_{0 \leq t \leq \infty} |\resolop_{0}\mu(X_{t}) - \resolop_{0}\nu(X_{t})|^{2} \right] \\
    & \leq    
    8E_x\!\left[ |A_{\infty} - B_{\infty}|^{2} \right]
    + 
    2 \|\resolop_{0}\mu - \resolop_{0}\nu\|_{\infty}^{2} \\
    & \leq    
    16 (\|\resolop_{0}\mu\|_{\infty} + \|\resolop_{0}\nu\|_{\infty}) \|\resolop_{0}\mu - \resolop_{0}\nu\|_{\infty} 
    + 
    2 (\|\resolop_{0}\mu\|_{\infty} + \|\resolop_{0}\nu\|_{\infty}) 
    \|\resolop_{0}\mu - \resolop_{0}\nu\|_{\infty}^{2} \\ 
    & \leq
    18 (\|\resolop_{0}\mu\|_{\infty} + \|\resolop_{0}\nu\|_{\infty}) \|\resolop_{0}\mu - \resolop_{0}\nu\|_{\infty} 
  \end{align}
  which completes the proof.
\end{proof}

To extend the above result to more general smooth measures,
we use a killing technique.
Let $\lambda$ be the exponential distribution with mean $1$, i.e., $\lambda((v, \infty)) = e^{-v}$ for all $v \geq 0$.
We set $\tilde{\Omega} \coloneqq \Omega \times [0, \infty)$,
$\tilde{\sigalg} \coloneqq \sigalg \otimes \Borel([0, \infty))$,
and $\tilde{P}_{x} \coloneqq P_{x} \otimes \lambda$.
Fix $\alpha > 0$.
We define $T_{\alpha} \colon \Omega \times [0, \infty) \to [0, \infty)$
by setting $T_{\alpha}(\omega, v) = \alpha^{-1}v$,
and define, for each $t \in [0, \infty]$ and $(\omega, v) \in \tilde{\Omega}$,
\begin{equation}
  \tilde{X}^{(\alpha)}_{t}(\omega, v)
  \coloneqq
  \begin{cases}
    X_{t}(\omega),& t < T_{\alpha}(\omega, v),\\
    \Delta, & t \geq T_{\alpha}(\omega, v),
  \end{cases}
  \quad 
  \tilde{\theta}^{\alpha}_{t}(\omega, v)
  \coloneqq 
  (\theta_{t}(\omega), (v- \alpha t) \vee 0).
\end{equation}
Then, 
$\tilde{X}^{(\alpha)} = (\tilde{\Omega},\allowbreak \tilde{\sigalg},\allowbreak (\tilde{X}^{(\alpha)}_{t})_{t \in [0, \infty]},
  \allowbreak (\tilde{P}_{x})_{x \in S_{\Delta}}, \allowbreak (\tilde{\theta}^{\alpha}_{t})_{t \in [0, \infty]})$ 
is a standard process (cf.\ \cite[Theorem~A.3.13]{Chen_Fukushima_12_Symmetric}).
(NB. In \cite[Theorem~A.3.13(ii)]{Chen_Fukushima_12_Symmetric}, a more general transformation is considered
and hence the measurability of $x \mapsto \tilde{P}_{x}(\cdot)$ becomes weak, so-called universal measurability. 
However, in our setting, it is easy to check that the map is still $\Borel(S)$-measurable.)
It is elementary to check that
$\tilde{X}^{(\alpha)}$ satisfies Assumption~\ref{assum: heat kernel}.
In particular,
the transition density of $\tilde{X}^{(\alpha)}$ is given by $\tilde{p}^{(\alpha)}_{t}(x,y) = e^{-\alpha t}p_{t}(x,y)$.
As a consequence,
the $\beta$-potential density of $\tilde{X}^{(\alpha)}$ is $\resol_{\alpha+\beta}$.
This implies that if $\mu$ is a smooth measure of $X$,
then $\mu$ is also smooth with respect to $\tilde{X}^{(\alpha)}$. 
In the following arguments,
any function $Y$ defined on $\Omega$ is naturally identified with the function $\tilde{Y}$ on $\tilde{\Omega}$
given by $\tilde{Y}(\omega, v) = Y(\omega)$.

\begin{lem} \label{lem: killed PCAF}
  Let $\mu$ be a smooth measure of $X$
  and $A$ be the associated PCAF of $X$.
  Fix $\alpha > 0$.
  We define, for each $t \in [0, \infty]$ and $(\omega, v) \in \tilde{\Omega}$,
  \begin{equation}
    \tilde{A}^{(\alpha)}_{t}(\omega, v) 
    \coloneqq 
    \begin{cases}
      A_{t}(\omega), & t < T_{\alpha}(\omega, v),\\
      A_{T_{\alpha}(\omega, v)}(\omega), & t \geq T_{\alpha}(\omega, v).
    \end{cases}
  \end{equation}
  Then $(\tilde{A}^{(\alpha)}_{t})_{t \geq 0}$ is a PCAF of $\tilde{X}^{(\alpha)}$
  and its associated smooth measure with respect to $\tilde{X}^{(\alpha)}$ is $\mu$.
\end{lem}

\begin{proof}
  It is easy to see that $(\tilde{A}^{(\alpha)}_{t})_{t \geq 0}$ is a PCAF of $\tilde{X}^{(\alpha)}$.
  We have that, 
  for every non-negative measurable function $f$ on $S$ and $\beta > 0$,
  \begin{align}
    \tilde{E}_{x}
    \left[
      \int_{0}^{\infty} e^{-\beta t} f(\tilde{X}^{(\alpha)}_{t})\, d\tilde{A}^{(\alpha)}_{t}
    \right]
    &=
    \tilde{E}_{x}
    \left[
      \int_{0}^{T_{\alpha}} e^{-\beta t} f(X_{t})\, dA_{t}
    \right] \\
    &=
    \int_{0}^{\infty}
    e^{-\beta t}
    E_{x}
    \left[
      \int_{0}^{s}  f(X_{t})\, dA_{t}
    \right]
    \alpha e^{-\alpha s}\, ds \\
    &=
    E_{x}
    \left[
      \int_{0}^{\infty}
      e^{-\beta t}
      f(X_{t})
      \int_{t}^{\infty} 
      \alpha e^{-\alpha s}\, ds\, dA_{t}
    \right]  \\
    &=
    E_{x}
    \left[
      \int_{0}^{\infty}
       e^{- (\alpha + \beta) t} f(X_{t})\, dA_{t}
    \right] \\
    &=
    \int_{S} \resol_{\alpha+\beta}(x, y) f(y)\, \mu(dy),
  \end{align}  
  where we use \eqref{2. eq: Revuz correspondence} to obtain the last equality.
  Hence, we deduce the desired result.
\end{proof}

Using the above lemma,
we can prove Theorem~\ref{3. thm: key inequality} as follows.

\begin{proof} [{Proof of Theorem~\ref{3. thm: key inequality}}]
  By regarding $\mu$ and $\nu$ as smooth measures of $\tilde{X}^{(\alpha)}$,
  we define $\tilde{A}^{(\alpha)}$ and $\tilde{B}^{(\alpha)}$ to be the associated PCAFs.
  Using Lemma~\ref{lem: killed PCAF}, we deduce that, for any $x \in S$, 
  \begin{align}
    E_x\!\left[\sup_{0 \leq t \leq T} |A_{t} - B_{t}|^{2} \right]
    & \leq  
    \tilde{E}_{x}\left[ \sup_{0 \leq t \leq T} |\tilde{A}^{(\alpha)}_{t} - \tilde{B}_{t}^{(\alpha)}|^{2} \right] 
    +
    \tilde{E}_{x} \left[ \sup_{0 \leq t \leq T} |A_{t} - B_{t}|^{2} ; T_{\alpha} < T\right] \\
    & \leq  
    \tilde{E}_{x}\left[ \sup_{0 \leq t \leq T} |\tilde{A}^{(\alpha)}_{t} - \tilde{B}_{t}^{(\alpha)}|^{2} \right] 
    + 
    2(1-e^{-\alpha T}) (E_{x}[A_{T}^{2}] + E_{x}[B_{T}^{2}]).
  \end{align}
  We have that 
  \begin{align}
    E_{x}[A_{T}^{2}]
    & =  
    E_x\!\left[ \left(\int_{0}^{T} dA_{t}\right)^{2} \right] \\
    & \leq 
    e^{2T} 
    E_x\!\left[ \left(\int_{0}^{\infty} e^{-t} dA_{t}\right)^{2} \right] \\
    &=
    2e^{2T}
    E_{x}
    \left[
      \int_{0}^{\infty} e^{-2t} 
      E_{X_{t}}\left[\int_{0}^{\infty} e^{-s}\, dA_{s}\right]\, 
      dA_{t}
    \right] \\
    &=
    2e^{2T}
    \int_{S} \resol_{2}(x,y) \int_{S} \resol_{1}(y, z)\, \mu(dz)\, \mu(dy)\\
    &\leq  
    2e^{2T}\supnorm{\resolop_{1}\mu}^{2},
  \end{align}
  where we use \cite[Exercise 4.1.7]{Chen_Fukushima_12_Symmetric} and \eqref{2. eq: Revuz correspondence} to obtain
  the second and third equalities, respectively.
  Similarly, we deduce that $E_{x}[B_{T}^{2}] \leq 2e^{2T} \|\resolop_{1}\nu\|_{\infty}^{2}$.
  Since $\resol_{\alpha}$ is the $0$-potential density of $\tilde{X}^{(\alpha)}$,
  it follows from Lemma~\ref{lem: L^2 inequality for 0-order measure} that 
  \begin{equation}
    \sup_{x \in S}
    \tilde{E}_{x}\left[ \sup_{0 \leq t \leq T} |\tilde{A}^{(\alpha)}_{t} - \tilde{B}_{t}^{(\alpha)}|^{2} \right] 
    \leq
    18 (\supnorm{\resolop_{\alpha}\mu} + \|\resolop_{\alpha}\nu\|_{\infty}) \|\resolop_{\alpha} \mu - \resolop_{\alpha}\nu\|_{\infty},
  \end{equation}
  which completes the proof.
\end{proof}


\subsection{Proof of Theorems \ref{1. thm: the main strong result} and \ref{1. thm: the main general result}}

In this subsection, we prove Theorems \ref{1. thm: the main strong result} and \ref{1. thm: the main general result}.
The key observation is that,
by Proposition~\ref{2. prop: maximum principle for difference of potentials},
Assumption~\ref{1. assum: stronger assumption} implies the convergence of the $\alpha$-potentials with respect to the supremum norm.
This is described precisely as follows.

\begin{lem} \label{4. lem: uniform convergence of potentials}
  Under Assumption~\ref{1. assum: stronger assumption},
  for any $\alpha > 0$,
  \begin{gather}
    \lim_{n \to \infty} 
    \supnorm{\resolop_{\alpha} \mu_{n} - \resolop_{\alpha} \mu}
    = 0,
    \qquad
    \sup_{n \geq 1} \supnorm{\resolop_{\alpha}\mu_{n}} 
    < \infty.
  \end{gather}
\end{lem}

\begin{proof}
  Fix $\alpha > 0$.
  Using Lemma~\ref{lem: resolvent identity for potentials},
  we deduce that 
  \begin{align}
    \lefteqn{|\resolop_{\alpha}\mu_{n}(x) - \resolop_{\alpha}\mu(x)|}\\
    &\leq     
    |\resolop_{1}\mu_{n}(x) - \resolop_{1}\mu(x)|
    + 
    |1 - \alpha| 
    \int_{S} \resol_{\alpha}(x,y) |\resolop_{1}\mu_{n}(y) - \resolop_{1}\mu(y)|\, m(dy)\\
    & \leq 
    \|\resolop_{1}\mu_{n} - \resolop_{1}\mu\|_{\infty}  
    + 
    |1-\alpha| \alpha^{-1} \|\resolop_{1}\mu_{n} - \resolop_{1}\mu\|_{\infty},
  \end{align}
  where we use that $\int_{S} \resol_{\alpha}(x,y)\, m(dy) \leq \alpha^{-1}$ at the last inequality.
  Since we have from Assumption~\ref{1. assum: stronger assumption}
  and Proposition~\ref{2. prop: maximum principle for difference of potentials} that 
  \begin{equation} \label{eq: convergence of 1-potentials}
    \lim_{n \to \infty} 
    \|\resolop_{1}\mu_{n} - \resolop_{1}\mu\|_{\infty} 
    = 0,
  \end{equation}
  we obtain the first result.
  Since $\supnorm{\resolop_{1}\mu} < \infty$,
  we have from Lemma~\ref{lem: resolvent identity for potentials} that $\supnorm{\resolop_{\alpha}\mu} < \infty$ for any $\alpha > 0$.
  Similarly, $\supnorm{\resolop_{\alpha}\mu_{n}} < \infty$ for any $n \geq 1$ and $\alpha > 0$.
  These, combined with the first result, yield the second result.
\end{proof}

By Theorem~\ref{3. thm: key inequality} and Lemma~\ref{4. lem: uniform convergence of potentials},
we can prove Theorem~\ref{1. thm: the main strong result} as follows.

\begin{proof} [Proof of Theorem~\ref{1. thm: the main strong result}]
  Suppose that Assumption~\ref{1. assum: stronger assumption} is satisfied.
  It follows from Theorem~\ref{3. thm: key inequality} that, for any $\alpha > 0$,
  \begin{align}
    \sup_{x \in S}
    E_x\!\left[ \sup_{0 \leq t \leq T} |A_{n}(t) - A(t)|^{2} \right]
    & \leq 
    18(\|\resolop_{\alpha}\mu_{n}\|_{\infty} + \|\resolop_{\alpha}\mu\|_{\infty}) 
    \|\resolop_{\alpha} \mu_{n} - \resolop_{\alpha}\mu\|_{\infty} \\
    & \quad
    + 
    4e^{2T}(1- e^{-\alpha T})( \|\resolop_{1}\mu_{n}\|_{\infty}^{2} + \|\resolop_{1}\mu\|_{\infty}^{2})
  \end{align}
  Letting $n \to \infty$ and then $\alpha \to 0$ in the above inequality
  and using  Lemma~\ref{4. lem: uniform convergence of potentials},
  we obtain the desired result.
\end{proof}

Finally,
using Theorem~\ref{1. thm: the main strong result},
we obtain Theorem~\ref{1. thm: the main general result} as follows.

\begin{proof} [{Proof of Theorem~\ref{1. thm: the main general result}}]
  Suppose that Assumption~\ref{1. assum: general assumption} is satisfied.
  By Proposition \ref{2. prop: maximum principle for difference of potentials},
  similarly to \eqref{eq: convergence of 1-potentials},
  it holds that 
  \begin{equation}
    \lim_{n \to \infty} 
    \supnorm{\resolop_{1}\mu_{n}^{V_{k}} - \resolop_{1}\mu^{V_{k}}}
    = 0,
    \quad 
    \forall k \geq 1.
  \end{equation}
  Thus, the measures $\mu_{n}^{V_{k}}$ and $\mu^{V_{k}}$ satisfy Assumption~\ref{1. assum: stronger assumption} for each $k \geq 1$.
  Define $A_{n}^{V_{k}}(t) \coloneqq \int_{0}^{t} 1_{V_{k}}(X_{s})\, A_{n}(ds)$
  and $A^{V_{k}}(t) \coloneqq \int_{0}^{t} 1_{V_{k}}(X_{s})\, A(ds)$.
  By Lemma~\ref{2. lem: restriction of smooth measures},
  $A_{n}^{V_{k}}$ and $A^{V_{k}}$ are the PCAFs
  associated with $\mu_{n}^{V_{k}}$ and $\mu^{V_{k}}$, respectively.
  Therefore, Theorem~\ref{1. thm: the main strong result} implies that,
  for any $\varepsilon, T>0$ and $k \geq 1$,
  \begin{equation} \label{eq: truncated PCAFs convergence}
    \lim_{n \to \infty}
    \sup_{x \in S}
    E_x\!\left[ \sup_{0 \leq t \leq T} |A_{n}^{V_{k}}(t) - A^{V_{k}}(t)|^{2} \right]
    = 0.
  \end{equation}
  Fix $x \in S$ and $\varepsilon, T>0$.
  Recall from \eqref{2. eq: exit time} that $\tau_{V_{k}}$ denotes the first exit time of $V_{k}$ by $X$.
  If $\tau_{V_{k}} > T$, then $A_{n}^{V_{k}}(t) = A_{n}(t)$ and $A^{V_{k}}(t) = A(t)$ for all $t \in [0, T]$.
  Thus, we deduce that
  \begin{align}
    P_x\!\left( \sup_{0 \leq t \leq T} |A_{n}(t) - A(t)| > \varepsilon \right)
    &\leq     
    P_{x}(\tau_{V_{k}} \leq T) 
    + 
    P_x\!\left( \sup_{0 \leq t \leq T} |A_{n}^{V_{k}}(t) - A^{V_{k}}(t)| > \varepsilon \right) \\
    &\leq 
    P_{x}(\tau_{V_{k}} \leq T)  
    + 
    \varepsilon^{-2} E_x\!\left[ \sup_{0 \leq t \leq T} |A_{n}^{V_{k}}(t) - A^{V_{k}}(t)|^{2} \right],
  \end{align}
  where we use Markov's inequality to obtain the last inequality.
  Since the process $X$ is conservative,
  $\tau_{V_{k}} \to \infty$ as $k \to \infty$ almost surely.
  Therefore, we can conclude that
  \begin{equation} \label{pr: desired conv of general PCAFs}
    \lim_{n \to \infty}
    P_x\!\left( \sup_{0 \leq t \leq T} |A_{n}(t) - A(t)| > \varepsilon \right)
    = 0.
  \end{equation}
\end{proof}

As seen in the above proof,
if the convergence $\lim_{k \to \infty}P_x(\tau_{V_k} \leq T) = 0$ holds locally uniformly in $x$,
then the desired convergence~\eqref{pr: desired conv of general PCAFs} holds locally uniformly in $x$.
For example, this assumption is satisfied when the process $X$ is a doubly-Feller process,
and we state this precisely below.
Recall from \cite{Chung_86_DoublyFeller} that 
the transition semigroup $(T_t)_{t > 0}$ of $X$ is said to have the \emph{Feller property} if 
it satisfies the following conditions:
\begin{enumerate} [label = (F\arabic*), leftmargin = *]
  \item for each $t > 0$ and continuous function $f$ on $S_\Delta$ with $f(\Delta) = 0$,
    $T_t f$ is continuous on $S_\Delta$;
  
  \item for each continuous function $f$ on $S_\Delta$ with $f(\Delta) = 0$, 
    $\displaystyle \lim_{t \downarrow 0} T_tf(x) = f(x)$;
\end{enumerate}
The semigroup $(T_t)_{t > 0}$ is said to have the \emph{strong Feller property} if 
\begin{enumerate} [label = (SF)] 
  \item for each $t>0$ and bounded measurable function $f$ on $S$,
    $T_t f$ is a bounded continuous function on $S$.
\end{enumerate}
If the semigroup has both the Feller property and the strong Feller property,
then it is said to have the \emph{doubly-Feller property}.

\begin{thm} \label{thm: modification for doubly-feller proc}
  Suppose that Assumption~\ref{1. assum: general assumption} is satisfied
  and $X$ is conservative.
  Moreover, suppose that the transition semigroup of $X$ has the doubly-Feller property.
  Then, for any $\varepsilon, T> 0$ and compact subset $K \subseteq S$,
  \begin{equation} \label{thm eq: locl unif convergence of PCAFs in probability}
    \lim_{n \to \infty}
    \sup_{x \in K}
    P_x\! 
    \left(
      \sup_{0 \leq t \leq T}|A_{n}(t) - A(t)| > \varepsilon
    \right)
    =0.
  \end{equation}
\end{thm}

\begin{proof}
  For each $k \geq 1$, we define a transition semigroup $(T_t^k)_{t > 0}$ with state space $S_\Delta$ 
  to be the semigroup of $X$ killed when exiting $V_k$,
  that is,
  \begin{equation}
    T_t^k f(x) \coloneqq E_x\!\left[ f(X_t) \cdot 1_{\{t < \tau_{V_k}\}} \right].
  \end{equation}
  As shown in the proof of \cite[Theorem in Section~1]{Chung_86_DoublyFeller},
  the semigroup $(T_t^k)_{t > 0}$ has the strong Feller property.
  Fix $t > 0$.
  Since it holds that 
  \begin{equation}
    P_x(\tau_{V_k} \leq t) = 1 - P_x(\tau_{V_k} > t) = 1 - T_t^k 1(x),
  \end{equation}
  the function $x \mapsto P_x(\tau_{V_k} \leq t)$ is continuous for each $k \geq 1$.
  Since $P_x(\tau_{V_k} \leq t) \to 0$ as $k \to \infty$ for each $x \in S$
  and $P_x(\tau_{V_k} \leq t)$ is non-increasing with respect to $k$ for each $x \in S$,
  Dini's theorem implies that 
  $P_x(\tau_{V_k} \leq t) \to 0$ as $k \to \infty$ locally uniformly in $x \in S$.
  Hence,
  \begin{equation}
    \lim_{k \to \infty} \sup_{x \in K} P_x(\tau_{V_k} \leq t) = 0.
  \end{equation}
  Using the above convergence instead of the almost-sure convergence $\tau_{V_k} \to \infty$ in the proof of Theorem~\ref{1. thm: the main general result},
  we obtain the desired result.
\end{proof}


\section{Sufficient conditions for Assumptions \ref{1. assum: stronger assumption} and \ref{1. assum: general assumption}} \label{sec: sufficient conditions}

In many examples of interest,
the transition density $p$ is jointly continuous.
Below, 
we provide sufficient conditions for Assumptions \ref{1. assum: stronger assumption} and \ref{1. assum: general assumption}
that are tractable in applications.
Note that, for a measure $\nu$ on $S$, we write $\supp(\nu)$ for the topological support of $\nu$.

\begin{thm} \label{1. thm: a sufficient condition for stronger assumption}
  Let $\mu, \mu_1, \mu_2, \ldots$ be finite Borel measures on $S$.
  Assume that the following conditions are satisfied.
  \begin{enumerate} [label = \textup{(A\arabic*)}]
    \item \label{1. thm item: stronger, continuity of heat kernel}
      The transition density $p$ is jointly continuous with values in $[0, \infty)$.
    \item \label{1. thm item: stronger, weak convergence of measures} 
      Each of the measures $\mu, \mu_{1}, \mu_{2}, \ldots$ charges no semipolar sets, and $\mu_{n} \to \mu$ weakly.
    \item \label{1. thm item: stronger, compactness of supports}
      There exists a compact subset $K_{0}$ of $S$ 
      such that $\supp(\mu_{n}) \subseteq K_{0}$ for all $n \geq 1$ and $\supp(\mu) \subseteq K_{0}$.
    \item \label{1. thm item: stronger, short time control}
      For any compact subset $K$ of $S$,
      \begin{equation}
        \lim_{t \to 0}
        \sup_{n \geq 1} 
        \sup_{x \in K}
        \int_{0}^{t} \int_{S} p_{s}(x, y)\, \mu_{n}(dy)\, ds
        = 0.
      \end{equation}
  \end{enumerate}
  Then Assumption~\ref{1. assum: stronger assumption} is satisfied.
\end{thm}

\begin{thm} \label{1. thm: a sufficient condition for general assumption}
  Let $\mu, \mu_1, \mu_2, \ldots$ be Radon measures on $S$.
  Assume that the following conditions are satisfied.
  \begin{enumerate} [label = \textup{(B\arabic*)}]
    \item \label{1. thm item: general, continuity of heat kernel}
      The transition density $p$ is jointly continuous with values in $[0, \infty)$.
    \item \label{1. thm item: general, vague convergence of measures} 
      Each of the measures $\mu, \mu_{1}, \mu_{2}, \ldots$ charges no semipolar sets, and $\mu_{n} \to \mu$ vaguely,
      i.e., for all compactly supported functions $f\colon S \to \RN$, 
      \begin{equation}
        \lim_{n \to \infty} 
        \int_{S} f(x)\, \mu_{n}(dx) 
        = 
        \int_{S} f(x)\, \mu(dx).
      \end{equation}
    \item \label{1. thm item: general, short time control}
      For any compact subset $K$ of $S$, it holds that 
      \begin{equation}
        \lim_{t \to 0}
        \sup_{n \geq 1} 
        \sup_{x \in K}
        \int_{0}^{t} \int_{K} p_{s}(x, y)\, \mu_{n}(dy)\, ds
        = 0.
      \end{equation}
  \end{enumerate}
  Then Assumption~\ref{1. assum: general assumption} is satisfied.
\end{thm}

We first prove Theorem~\ref{1. thm: a sufficient condition for stronger assumption}.
By Proposition \ref{2. prop: smoothness condition},
\ref{1. thm item: stronger, short time control} yields that $\mu_{n} \in \smooth_{00}$ for each $n \geq 1$.
The following result states that the limiting measure $\mu$ also belongs to $\smooth_{00}$.

\begin{lem} \label{4. lem: short time control for the limit}
  Under the conditions of Theorem~\ref{1. thm: a sufficient condition for stronger assumption},
  for any compact subset $K$
  \begin{equation}
    \lim_{\delta \downarrow 0} 
    \sup_{x \in K}
    \int_{0}^{\delta} \int_{S} p_{t}(x, y)\, \mu(dy)\, dt
    =0.
  \end{equation}
  In particular, $\mu \in \smooth_{00}$.
\end{lem}

\begin{proof}
  By conditions \ref{1. thm item: stronger, weak convergence of measures} and \ref{1. thm item: stronger, short time control},
  we have that, for each $t > 0$ and $x \in S$,
  \begin{equation}
    \lim_{n \to \infty} 
    \int_{S} p_{t}(x, y)\, \mu_{n}(dy)
    = 
    \int_{S}  p_{t}(x,y)\, \mu(dy).   
  \end{equation}
  Thus, 
  using Fatou's lemma 
  and condition \ref{1. thm item: stronger, short time control}, 
  we deduce that, for any compact subset $K$ of $S$,
  \begin{align}
    \lim_{\delta \downarrow 0} 
    \sup_{x \in K}
    \int_{0}^{\delta} \int_{S} p_{t}(x, y)\, \mu(dy)\, dt
    & \leq
    \lim_{\delta \downarrow 0} 
    \sup_{x \in K}
    \liminf_{n \to \infty}
    \int_{0}^{\delta} \int_{S} p_{t}(x, y)\, \mu_{n}(dy)\, dt \\
    & \leq
    \lim_{\delta \downarrow 0} 
    \liminf_{n \to \infty}
    \sup_{x \in K}
    \int_{0}^{\delta} \int_{S} p_{t}(x,y) \mu_{n}(dy)\, dt\\
    & = 
    0.
  \end{align}
  The last assertion follows from Proposition \ref{2. prop: smoothness condition}.
\end{proof}

Now, we are ready to show Theorem~\ref{1. thm: a sufficient condition for stronger assumption}.

\begin{proof} [{Proof of Theorem~\ref{1. thm: a sufficient condition for stronger assumption}}]
  Suppose conditions \ref{1. thm item: stronger, continuity of heat kernel}, \ref{1. thm item: stronger, weak convergence of measures},
  \ref{1. thm item: stronger, compactness of supports}, and \ref{1. thm item: stronger, short time control} 
  of Theorem~\ref{1. thm: a sufficient condition for stronger assumption} 
  are satisfied.
  We have from \ref{1. thm item: stronger, short time control} and Proposition \ref{2. prop: smoothness condition} 
  that $\mu_{n} \in \smooth_{00}$ for each $n \geq 1$ and 
  from Lemma~\ref{4. lem: short time control for the limit} that $\mu \in \smooth_{00}$.
  By the weak convergence of $\mu_{n}$ to $\mu$
  and the equicontinuity of $\{p_{t}(x,y)\}_{x \in K_{0}}$ as continuous functions of $(t, y) \in [\delta, T] \times K_{0}$,
  we deduce that, for any $\delta, T > 0$,
  \begin{equation} \label{eq: convergence of energy for finite time interval}
    \lim_{n \to \infty}
    \sup_{x \in K_{0}}
    \left|
      \int_{S} \int_{\delta}^{T} e^{-s} p_{s}(x,y)\, ds\, \mu_{n}(dy)
      -
      \int_{S} \int_{\delta}^{T} e^{-s} p_{s}(x,y)\, ds\, \mu(dy)
    \right|
    = 0
  \end{equation}
  (this is deduced from the Skorohod representation theorem; see also \cite[Lemma 1.1]{Whitt_unpublished_internet_supplement}).
  By \ref{1. thm item: stronger, short time control},
  we can find $t_0 > 0$ such that 
  \begin{equation}
    \sup_{n \geq 1} \sup_{x \in K_0} \int_0^{t_0} \int_S p_s(x,y)\, \mu_n(dy)\, ds < 1.
  \end{equation}
  This, combined with \eqref{pr: Maximal principle for potential}, \eqref{pr: potential bound by CK equation}, and \ref{1. thm item: stronger, compactness of supports},
  yields that $\sup_{n \geq 1}\|\resolop_1 \mu_n\|_\infty < \infty$.
  Using the Chapman--Kolmogorov equation~\eqref{2. eq: C-K equation},
  we obtain that 
  \begin{align}
    \sup_{x \in K_{0}}
    \int_{T}^{\infty} \int_{S} e^{-t} p_{t}(x, y)\, \mu_{n}(dy)\, dt 
    &=
    \sup_{x \in K_{0}}
    e^{-T}\int_{0}^{\infty} \int_S e^{-t} p_{T+ t}(x, y)\,  \mu_{n}(dy)\, dt\\
    &= 
    \sup_{x \in K_{0}}
    e^{-T}\int_{0}^{\infty} \int_S e^{-t} \int_S p_T(x,z) p_t(z,y)\, m(dz)\, \mu_{n}(dy)\, dt\\
    &\leq
    e^{-T} \|\resolop_1 \mu_n\|_\infty
    \sup_{x \in K_{0}}
    \int_S p_T(x,z)\, m(dz)\\
    &\leq     
    e^{-T} \|\resolop_1 \mu_n\|_\infty.
  \end{align}
  Thus, 
  \begin{equation} \label{eq: uniform vanishment of energy at time infinity}
    \lim_{T \to \infty} 
    \sup_{n \geq 1} 
    \sup_{x \in K_{0}} 
    \int_{T}^{\infty} \int_{S} e^{-t} p_{t}(x, y)\, \mu_{n}(dy)\, dt   
    = 0.
  \end{equation}
  Similarly, it holds that 
  \begin{equation} \label{eq: vanishment of energy at time infinity}
    \lim_{T \to \infty}  
    \sup_{x \in K_{0}} 
    \int_{T}^{\infty} \int_{S} e^{-t} p_{t}(x, y)\, \mu(dy)\, dt  
    = 0.
  \end{equation}
  The triangle inequality yields that, for any $x \in S$ and $\delta, T > 0$ with $\delta < T$,
  \begin{align}
    |\resolop_{1}\mu_{n}(x) - \resolop_{1}\mu(x)| 
    &\leq   
    \left|
      \int_{S} \int_{0}^{\infty} e^{-t} p_{t}(x,y)\, dt\, \mu_{n}(dy)
      -
      \int_{S} \int_{0}^{\infty} e^{-t} p_{t}(x,y)\, dt\, \mu(dy)
    \right| \\
    &\leq     
    \int_{0}^{\delta} \int_{S} p_{t}(x, y)\, dt\, \mu_{n}(dy) 
    + 
    \int_{0}^{\delta} \int_{S} p_{t}(x, y)\, dt\, \mu(dy) \\
    &\qquad
    + 
    \left|
      \int_{S} \int_{\delta}^{T} e^{-t} p_{t}(x,y)\, dt\, \mu_{n}(dy)
      -
      \int_{S} \int_{\delta}^{T} e^{-t} p_{t}(x,y)\, dt\, \mu(dy)
    \right| \\
    &\qquad 
    + 
    \int_{T}^{\infty} \int_{S} e^{-t} p_{t}(x, y)\, \mu_{n}(dy)\, dt  
    + 
    \int_{T}^{\infty} \int_{S} e^{-t} p_{t}(x, y)\, \mu(dy)\, dt.
  \end{align}
  Taking the supremum over $x \in K_{0}$ in the above inequality,
  letting $n \to \infty$, $\delta \to 0$, and $T \to \infty$ in the above inequality,
  and using \ref{1. thm item: stronger, short time control},
  Lemma~\ref{4. lem: short time control for the limit},
  \eqref{eq: convergence of energy for finite time interval}, \eqref{eq: uniform vanishment of energy at time infinity},
  and \eqref{eq: vanishment of energy at time infinity},
  we deduce that 
  \begin{equation}
    \lim_{n \to \infty} 
    \sup_{x \in K_{0}}
    |\resolop_{1}\mu_{n}(x) - \resolop_{1}\mu(x)|
    = 0.
  \end{equation}
  Therefore, Assumption~\ref{1. assum: stronger assumption} is satisfied.
\end{proof}

Next, we prove Theorem~\ref{1. thm: a sufficient condition for general assumption}.
By using the following result,
Theorem~\ref{1. thm: a sufficient condition for general assumption} is easily deduced from Theorem~\ref{1. thm: a sufficient condition for stronger assumption}.

\begin{lem} \label{4. lem: choice of nest}
  Under the assumptions of Theorem~\ref{1. thm: a sufficient condition for stronger assumption},
  there exists an increasing sequence $(V_{k})_{k \geq 1}$ of relatively compact open subsets of $S$ with $\bigcup_{k \geq 1} V_{k} = S$
  such that $\mu_{n}^{V_{k}} \to \mu^{V_{k}}$ weakly as $n \to \infty$, for each $k \geq 1$.
\end{lem}

\begin{proof}
  By \cite[Theorem 1]{Williamson_Janos_87_Construction},
  there exists a metric $d$ on $S$ inducing the topology on $S$ such that any bounded closed subset of $S$ is compact.
  Fix $\rho \in S$.
  Let $(r_{k})_{k \geq 1}$ be an increasing sequence with $r_{k} \uparrow \infty$ 
  such that $B_{d}(\rho, r_{k}) \coloneqq \{x \in S \mid d(\rho, x) < r_{k}\}$ is relatively compact 
  and $\mu(\partial B_{d}(\rho, r_{k})) = 0$.
  Define $V_{k} \coloneqq B_{d}(\rho, r_{k})$. 
  Condition \ref{1. thm item: general, vague convergence of measures} of Theorem~\ref{1. thm: a sufficient condition for general assumption} implies that 
  $\mu_{n}^{V_{k}} \to \mu^{V_{k}}$ weakly.
  Hence, we complete the proof.
\end{proof}

\begin{proof} [{Proof of Theorem~\ref{1. thm: a sufficient condition for general assumption}}]
  Suppose that conditions \ref{1. thm item: general, continuity of heat kernel}, \ref{1. thm item: general, vague convergence of measures}, 
  and \ref{1. thm item: general, short time control} of Theorem~\ref{1. thm: a sufficient condition for general assumption} 
  are satisfied.
  Fix $(V_{k})_{k \geq 1}$ appearing in Lemma~\ref{4. lem: choice of nest}.
  Then, for each $k \geq 1$, the measures $\mu_{n}^{V_{k}}$ and $\mu^{V_{k}}$ satisfy
  the conditions \ref{1. thm item: stronger, continuity of heat kernel}, \ref{1. thm item: stronger, weak convergence of measures},
  \ref{1. thm item: stronger, compactness of supports}, and \ref{1. thm item: stronger, short time control} 
  of Theorem~\ref{1. thm: a sufficient condition for stronger assumption}
  with $K_{0} = \overline{V}_{k}$.
  Thus, we have from Theorem~\ref{1. thm: a sufficient condition for stronger assumption} that, for each $k \geq 1$,
  $\mu_{n}^{V_{k}} \in \smooth_{00}$ for all $n \geq 1$, $\mu^{V_{k}} \in \smooth_{00}$, and 
  \begin{equation} 
    \lim_{n \to \infty} 
    \inftynorm{\resolop_{1}\mu_{n}^{V_{k}} - \resolop_{1}\mu^{V_{k}}}{\mu_{n}^{V_{k}} + \mu^{V_{k}}}
    =0.
  \end{equation}
  Thus, Assumption~\ref{1. assum: general assumption} is verified.
\end{proof}

\appendix

\section*{Acknowledgement}
I would like to thank my supervisor Dr.\ David Croydon for his support and fruitful discussions,
and Dr.\ Naotaka Kajino and Dr.\ Takahiro Mori for their valuable comments 
that helped me improve the earlier version of this paper.
I am also grateful to the anonymous referee for comments on the domination principle for $\alpha$-potentials,
which make the results applicable to the dual setting beyond the symmetric Dirichlet form framework.
This work was supported by 
JSPS KAKENHI Grant Number JP24KJ1447
and 
the Research Institute for Mathematical Sciences, 
an International Joint Usage/Research Center located in Kyoto University.

\bibliographystyle{unsrturl}
\bibliography{ref_convergence_of_PCAF}
\end{document}